\documentclass[11pt]{article}
\usepackage[a4paper,
            bindingoffset=0.2in,
            left=0.75in,
            right=0.8in,
            top=1in,
            bottom=1.5in,
            footskip=.5in]{geometry}

\usepackage{setspace}
\usepackage{color}
\definecolor{aleacolour}{rgb}{0.09,0.32,0.44} 
\usepackage[pdfborder={0 0 0},pdfborderstyle={},colorlinks=true,linkcolor=aleacolour,citecolor=red,urlcolor=blue]{hyperref}
\usepackage{setspace, ifdraft, blindtext,comment}
\ifdraft{\doublespace}{\singlespace} 
\ifdraft{
    \usepackage{todonotes}}{
    \usepackage[disable]{todonotes}
}

\usepackage{mathtools}
\usepackage{amsmath, amssymb, amsthm}
\usepackage{cases}
\usepackage{float}

\usepackage[utf8]{inputenc}

\usepackage{kantlipsum}
\usepackage{tikz}
\usepackage{pgf,tikz,svg}
\usetikzlibrary{arrows}
\usetikzlibrary[patterns]
\usetikzlibrary{calc,decorations.pathreplacing,positioning}

\usepackage{amsmath, amssymb, amsthm}
\usepackage{color}

\theoremstyle{plain}
\newtheorem{thm}{Theorem}[section]

\newtheorem*{prop*}{Proposition}
\newtheorem{lem}[thm]{Lemma}
\newtheorem*{lem*}{Lemma}

\theoremstyle{definition}
\newtheorem{definition}[thm]{Definition}


\newcommand{\beq}{\begin{equation}}
\newcommand{\eeq}{\end{equation}}
\date{}
\title{On the growth rate of the Stanley-Wilf limit of blockable permutations}
\author{
Saksham Sethi \thanks{North Carolina School of Science and Mathematics, Durham, NC, 27705. Email:
\href{mailto:saksham.sethi.123.123@gmail.com} {\nolinkurl{saksham.sethi.123.123@gmail.com}}.}
\and 
Fan Wei\thanks{Department of Mathematics, Duke University, Durham, NC 27710, USA. Supported by NSF grant DMS-2401414. Email: \href{mailto:fan.wei@duke.edu}
{\nolinkurl{fan.wei@duke.edu}}.}
}
\begin{document}

\maketitle

\begin{abstract}
Given a permutation $\pi$, let $\text{Av}_n(\pi)$ be the number of permutations of length $n$ that avoid $\pi$ as a subpermutation. The celebrated resolution of the Stanley-Wilf conjecture by Marcus and Tardos confirmed that the limit $L(\pi) = \lim_{n \to \infty} |\text{Av}_n(\pi)|^{1/n}$ exists. A central and challenging question concerns the behavior of $L(\pi)$ as a function of the pattern length $|\pi|$. While Fox proved that $L(\pi)$ is exponential in $|\pi|$ for almost all permutations, it is known that $L(\pi)$ grows polynomially for specific structural classes. For instance, $L(\pi)$ is known to be quadratic in $|\pi|$ when $\pi$ is a monotone or a layered permutation. In this paper, we address this question for {\it blockable} permutations $\pi$.    
\end{abstract}

\section{Introduction}

Permutation pattern avoidance has become an active topic in modern combinatorics, with deep connections to computer science, algebra, and geometry. The subject traces back to the work of Knuth \cite{knuth_1997}, Tarjan \cite{tarjan_1972}, and Pratt \cite{pratt_1973} on sorting with stacks and queues, later extended by Avis and Newborn \cite{avis_newborn_1981}, Rosentiehl and Tarjan \cite{rosenstiehl_tarjan_1986}, and Felsner and Pergel \cite{felsner_pergel_2013}. Sorting networks and related structures remain influential in algorithm design and parallel computation \cite{valsalam_miikkulainen_2013}.

Beyond its computational origins, permutation avoidance has rich ties to algebraic geometry through Schubert varieties, to representation theory via Kazhdan–Lusztig polynomials, and to algebraic combinatorics through the work of Hamaker, Pawlowski, and Sagan \cite{hamaker_pawlowski_sagan_2020}, Bloom and Sagan \cite{bloom_sagan_2020}, and Marmor \cite{marmor_2025}. 

In this paper, we investigate the asymptotic growth of certain families of pattern-avoiding permutations and show that these classes grow only polynomially in size. These results provide new examples of avoidance families with unusually slow growth and address a long-standing question about the growth rates of permutation classes. We next review the relevant background and introduce the necessary definitions.

Denote $[n] = \{1, 2, \dots, n\}$. A permutation of length $n$ is called an $n$-permutation. We begin with a simple definition.

\begin{definition}
    An $n$-permutation $\sigma$ \textit{contains} a $k$-permutation $\pi$ if there exist integers $1 < x_1 < x_2 \dots < x_k < n$ such that for all $0 \leq i,j\leq k$ we have $\sigma(x_i) <\sigma(x_j)$ if and only if $\pi(x_i) < \pi(x_j)$. Otherwise, $\sigma$ avoids $\pi$.
\end{definition}

In other words, $\sigma$ contains $\pi$ if and only if there is a subsequence in $\sigma$ with the same relative order as $\pi$. For example, the permutation $A = 42153$ contains $B=312$ but avoids $C=123$.

\begin{definition}
    $\text{Av}(\pi)$ is the class of all permutations that avoid $\pi$, and $\text{Av}_n(\pi)$ is the class of all $n$-permutations that avoid $\pi$.
\end{definition}

The size of these avoidance classes has been of interest. As mentioned earlier in the introduction, these avoidance classes link the theory of permutations to algebraic geometry and representation theory. Furthermore, the size of special avoidance classes have connections to classical objects in combinatorics and theoretical computer science: the classes $\text{Av}(123)$ and $\text{Av}(132)$ correspond to Dyck paths and binary trees, the class $\text{Av}(231)$ corresponds to stack-sortable permutations first studied by Knuth \cite{knuth_1997}, and the avoidance class of all permutations that avoid both $2413$ and $3142$, denoted $\text{Av}(2413, 3142)$, correspond to large Schröder numbers.  
A famous result by Knuth \cite{knuth_1997} shows that $|\text{Av}_n(\pi)| = \frac{1}{n+1}\binom{2n}{n}$ for any $3$-permutation $\pi$. 

In 1980, Stanley and Wilf conjectured that the {\it Stanley-Wilf limit} \[L(\pi)=\lim_{n \to \infty} (|\text{Av}_n(\pi)|)^{1/n}
\] is finite. The first breakthrough in this conjecture was in 2000 by Klazar \cite{klazar_2000}, when he showed that the Furedi-Hajnal conjecture \cite{furedi_hajnal_1992} on the extremal function of permutation matrices, as described below, implies the Stanley-Wilf conjecture. 

\begin{definition}
    A binary matrix is a matrix with all entries either $0$ or $1$, and we call it a \textit{permutation matrix} if it has exactly one 1-entry per row and per column. A binary matrix $A$ \textit{contains} a $k \times l$ binary matrix $P$ if a submatrix $B$ of $A$ with dimensions $k \times l$ exists such that $B(i, j) = 1$ if $P(i, j) = 1$. Otherwise, $A$ \textit{avoids} $P$. 
\end{definition}

\begin{definition}
    The extremal function, $\text{ex}_P(n)$, of a binary matrix $P$ is the maximum number of ones in an $n \times n$ binary matrix avoiding $P$. 
\end{definition}

The Furedi-Hajnal conjecture hypothesized that for each permutation matrix $P$, we have $\text{ex}_P(n) = O(n)$. 
In 2004, Marcus and Tardos \cite{marcus_tardos_2004} proved this conjecture, particularly showing the following:

\begin{thm}[Marcus-Tardos Theorem \cite{marcus_tardos_2004}]
    For any $k$-permutation matrix $P$, we have $$\mathrm{ex}_P(n) \leq 2k^4 \binom{k^2}{k}n$$
\end{thm}

By the Marcus-Tardos theorem, we can define the constant \[c_P =\lim_{n \to \infty} \frac{\text{ex}_P(n)}{n},\] commonly  known as the \textit{Furedi-Hajnal limit} of $P$, also referred to as the extremal constant of $P$. We will use this interchangeably with $c(\pi)$ where $\pi$ is the permutation represented by permutation matrix $P$. 

From Klazar's \cite{klazar_2000} equivalence, it was known that $L(\pi)=2^{O(c(\pi))}$, which shows the doubly-exponential bound $L(\pi) \leq 15^{2k^4 \binom{k^2}{k}}$, proving the Stanley-Wilf conjecture. More recently,  Cibulka \cite{cibulka_2009} showed these two limits are polynomially related, i.e., \[L(\pi) = O(c(\pi)^2).\] 

Since the result of Marcus and Tardos, many efforts have been focused on finding bounds for $L(\pi)$. It was conjectured by many prominent researchers after 2004 that $L(\pi)=\Theta(k^2)$ for every $k$-permutation $\pi$. It was also widely believed that a \textit{layered} permutation, one that is formed by a concatenation of decreasing sequences where the numbers in each sequence are smaller than the numbers in the following sequences, had the largest Stanley-Wilf limit. It was also shown by Claesson, Jelinek, and Steingrimsson \cite{claesson_jelinek_steingrimsson_2011} that $L(\pi) \leq 4k^2$ for layered $k$-permutations $\pi$, so it would follow from the belief that the Stanley-Wilf limit is indeed quadratic in the size of $k$. 

However, in 2013, Fox \cite{fox_2011} disproved these conjectures, showing that for each $k$ there is a $k$-permutation $\pi$ with $L(\pi)=2^{\Omega(k^{1/4})}$. Furthermore, he even showed that asymptotically almost all $k$-permutations $\pi$ have $L(\pi)$ exponential in $k$. 

Demonstrating for which classes of permutations there are subexponential bounds on Stanley-Wilf limits has been one of the major interests:

\

\noindent\textbf{Main Open Question.} 
For which families of permutations $\mathcal{P}$ is the Stanley–Wilf limit of $\text{Av}(\pi)$, for $\pi \in \mathcal{P}$, polynomial in the size of the pattern $\pi$?

\

This question has implication in the query complexity of permutation property testing (see e.g., \cite{fox_wei_2018} and \cite{fox_wei_2017}). Property testing is a prominent field of theoretical computer science that asks for algorithms that can efficiently distinguish, using few queries, whether a combinatorial object (a permutation in this case) satisfies a property or is far from satisfying it. Variations of this main open question, particularly when analyzing the growth rate of $\text{Av}_n(\prod)$ where $\prod$ is a \textit{set} of forbidden permutations, have been thoroughly studied in \cite{simion_schmidt_1985}, \cite{west_1996}, and \cite{albert_atkinson_2007} where a full categorization of polynomial growth has been shown as long as $|\prod| \geq 2$. However, in the case of $|\prod|=1$, namely when we are dealing with a principal class, very few families are known to have polynomial (even subexponential) Stanley-Wilf limits. This case has been of significant interest as mentioned in \cite{cibulka_kyncl_2017}. 

We will first briefly introduce a principal class that is known to have a polynomial Stanley-Wilf limit. The direct sum and skew sum operations are crucial, as defined below. 

\begin{definition}
    For two permutation matrices $P_1$ and $P_2$ of arbitrary size, the {\it direct sum} $P_1 \oplus P_2$ is the block matrix $\begin{bmatrix}
        P_1 & 0 \\
        0 & P_2 \\
    \end{bmatrix}$, whereas the {\it skew sum} $P_1 \ominus P_2$ is the block matrix $\begin{bmatrix}
        0 & P_1 \\
        P_2 & 0 \\
    \end{bmatrix}$.
\end{definition}

One example of a specific family with a polynomial Stanley-Wilf limit is all permutations $P$ if they can be formed by a series of direct \textit{and} skew sums of smaller identity matrices. This is an easy generalization of layered matrices, which require that the identities can only be concatenated through one of direct sum \textit{or} skew sum.

We briefly review the proof of this generalization, mainly built off of the proof for layered permutations from Claesson, Jelinek, and Steingrimsson \cite{claesson_jelinek_steingrimsson_2011}. In this proof, a key definition is that of the \textit{merge} of two permutation classes.
 The \textit{merge} of two permutation classes $C$ and $D$, namely $C \odot D$, is the class of all permutations whose entries can be colored red and blue so that the red subsequence is order isomorphic to a permutation $C_1 \in C$ and the blue subsequence is order isomorphic to a permutation $D_1 \in D$.

It has been shown by Albert \cite{albert_pantone_vatter_2016} that the growth rate of a permutation class formed by a merge  can be bounded nicely as follows. 

\begin{lem}\label{lem:albert}
    $$\mathrm{gr}(C \odot D) \leq \left(\sqrt{\mathrm{gr}(C)}+\sqrt{\mathrm{gr}(D)}\right)^2$$
\end{lem}

Therefore, if $\text{Av}(\pi)$ can be written as a merge of two families (in which case this avoidance class is called \textit{splittable}), then we can apply Albert's lemma, to reduce the growth rate of $\text{Av}(\pi)$ to the two families. However, this is typically quite challenging. The most general theorem is a lemma by Jelinek and Valtr \cite{jelinek_valtr_2015}, which states that for all nonempty permutations $A,B,$ and $C$, we have 

\begin{equation} \label{oplus_ineq}
    \text{Av}(A \oplus B \oplus C) \subseteq \text{Av}(A \oplus B) \odot \text{Av}(B \oplus C)
\end{equation}
The same condition holds if we replace all $\oplus$ by $\ominus$. 

Therefore if $\pi$ can be obtained by a series of direct sum or skew sum of permutations $\pi_i$, then we can analyze $\text{Av}(\pi)$ by the growth rates of $\text{Av}(\pi_i)$. A corollary from \cite{vatter_2015} also tells us that the growth rate of any layered permutation $\beta$ of length $k$ and at most two layers, which recall is the permutation formed by a concatenation of at most two decreasing sequences where the numbers in the first sequence are smaller than the numbers in the second sequence, is $\text{Av}(\beta)=(k-1)^2$. By induction (using the mentioned base case and the previously mentioned inductive step), this gives us a polynomial bound for all permutations formed by a combination of direct and skew sum of {\it identity matrices.}  

However, this method doesn't generalize to a large category of permutations because it is highly specific to the case of direct/skew sum and does not easily extend. In fact, it will not even easily extend to the case where we have direct/skew sums of \textit{non-identity} matrices; this is because even though Condition \ref{oplus_ineq} will remain true, we will no longer have a base case to rely on for $\text{Av}(A \oplus B)$. More generally – even beyond the case of direct/skew sums – as noted in Vatter's survey \cite{vatter_2015}, the difficulty is that is that it is hard to say much about the splittability of classes such as $\text{Av}(\pi)$ where $\pi$ is not a permutation obtained by direct/skew sums.  

In summary, whenever $\text{Av}(\pi)$ is known to possess a suitable decomposition, such as a direct or skew-sum splitting, Albert’s bound (Lemma \ref{lem:albert}) applies. In contrast, determining whether $\text{Av}(\pi)$ is decomposable is, in general, a highly nontrivial problem, even for many structurally simple permutations $\pi$.

In this paper, we start by offering a new method that proves a polynomial Stanley-Wilf limit for a broad family of permutations that covers the generalization presented above and more, including those cases that are hard to tackle through splittability. We start with some definitions. 

\begin{definition}
    Given a permutation $\sigma$ of length $m$ and nonempty permutations $P_1,\dots,P_m$, respectively, the inflation of $\sigma$ by $(P_1,\dots,P_m)$, denoted $\sigma[P_1,\dots,P_m ]$ is the permutation of length $|P_1|+\dots+|P_m|$ obtained by replacing each entry $\sigma(i)$ by an interval that is order-isomorphic to $P_1$ such that the intervals themselves are order-isomorphic to $\sigma$. 
\end{definition}

\begin{definition}
    We call a permutation \textit{c-blockable} if it can be represented as an inflation $\sigma[P_1, \dots, P_c]$ for some permutation $\sigma$ and some  $P_1, \dots, P_c$. For a fixed set of permutation families $\mathcal{P}_1, \dots, \mathcal{P}_c$, let $B_{\{\mathcal{P}_1, \dots, \mathcal{P}_c\}}$ denote the family of all $c$-blockable permutations $\sigma[P_1, \dots, P_c]$ where $\sigma$ is any $c$-permutation and $P_i \in \mathcal{P}_i$.
\end{definition}

An example of an inflation is shown in Figure \ref{fig:partitioned-squares}, namely of the case $2413[1, 132, 321, 12]$. This is a 4-blockable permutation. Notice that this permutation cannot be written as a sequence of direct sum and skew sum of smaller permutations. 

\begin{figure}[htbp]
\centering
\begin{tikzpicture}[scale=0.55,line cap=round,line join=round]
  \draw[line width=2pt] (0,0) rectangle (9,9);

  \def\xA{0} 
  \def\xB{1} 
  \def\xC{4}  
  \def\xD{7} 

  \def\yPthree{0}        
  \def\yPone{3}         
  \def\yPfour{4}         
  \def\yPtwo{6}

  \filldraw[fill=gray!30, line width=1pt] (\xA,\yPone) rectangle ++(1,1);
  \filldraw[black] (\xA+0.5,\yPone+0.5) circle (0.15);

  \filldraw[fill=gray!30, line width=1pt] (\xB,\yPtwo) rectangle ++(3,3);
  \filldraw[black] (\xB+0.5,\yPtwo+0.5) circle (0.15);
  \filldraw[black] (\xB+1.5,\yPtwo+2.5) circle (0.15);
  \filldraw[black] (\xB+2.5,\yPtwo+1.5) circle (0.15);

  \filldraw[fill=gray!30, line width=1pt] (\xC,\yPthree) rectangle ++(3,3);
  \filldraw[black] (\xC+0.5,\yPthree+2.5) circle (0.15);
  \filldraw[black] (\xC+1.5,\yPthree+1.5) circle (0.15);
  \filldraw[black] (\xC+2.5,\yPthree+0.5) circle (0.15);

  \filldraw[fill=gray!30, line width=1pt] (\xD,\yPfour) rectangle ++(2,2);
  \filldraw[black] (\xD+0.5,\yPfour+0.5) circle (0.15);
  \filldraw[black] (\xD+1.5,\yPfour+1.5) circle (0.15);

  \foreach \x in {1,...,8} {\draw[gray!50,thin] (\x,0)--(\x,9);}
  \foreach \y in {1,...,8} {\draw[gray!50,thin] (0,\y)--(9,\y);}
\end{tikzpicture}

\caption{A $4$-blockable permutation}
\label{fig:partitioned-squares}
\end{figure}
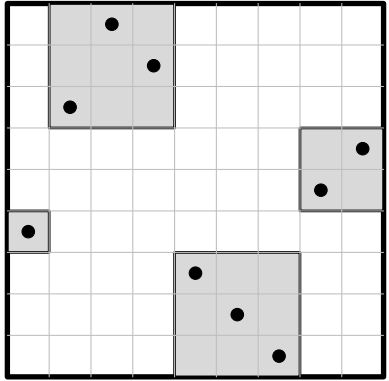

We now state our main result. Recall that for a permutation matrix $P$, we have $L(P) = L(\pi)$ where $\pi$ is the permutation corresponding to the one-entries in $P$.
\\

\begin{thm} \label{thm:mainresult}
    Consider some fixed positive integer $c$ and permutation families $\mathcal{P}_1, \dots, \mathcal{P}_c$. If, for all $1 \leq i \leq c$, the Stanley-Wilf limit of a permutation in $\mathcal{P}_i$ is polynomial in the size of the permutation, then for sufficiently large $k$ the Stanley-Wilf limit of a $k$-permutation in $B_{\{\mathcal{P}_1, \dots, \mathcal{P}_c\}}$ is polynomial in $k$. 

More precisely, if for all $1 \leq i \leq c$ and any permutation $P_i \in \mathcal{P}_i$, we have $L(P_i) \leq |P_i|^{a_i}$,  
let $a = \max_{1 \leq i \leq c} a_i$. Then, $$L(B_{\{\mathcal{P}_1, \dots, \mathcal{P}_c\}}) = O(k^{4a + 16c^2+64ac^2 \ln c});$$
\end{thm}

We prove this result as a direct consequence of Theorem \ref{thm: thm2.4}, 
and we made no attempt to optimize the constant. 

\section{$c$-blockable permutations}
In this section, we prove Theorem \ref{thm:mainresult}. Before that, we will introduce the framework used in this proof and its origins.

To help understand our approach, we will first quickly review the framework of Marcus-Tardos' approach for their bound on the extremal function. Consider a $k \times k$ permutation matrix $P$ and an $n \times n$ matrix $A$ that avoids $P$. Say that $A$ has $\text{ex}_P(n)$ one-entries. We split matrix $A$ into smaller blocks of dimension $n/k^2$ by $n/k^2$. A block is called \textit{wide} if it has at least $k$ nonempty columns and \textit{tall} if it has at least $k$ nonempty rows. By a pigeonhole argument, it is easy to show that there cannot be more than $k\binom{k^2}{k}\frac{n}{k^2}$ wide blocks and same for the tall blocks. So, we cannot have more than $k^3 \binom{k^2}{k}n$ ones in total considering all wide and tall blocks. For the blocks that are neither wide nor tall, we can simply bound the number of ones as $(k-1)^2\text{ex}_P(n/k^2)$. Finally, we obtain the recursion $$\text{ex}_P(n) \leq (k-1)^2\text{ex}_P(n/k^2) + 2k\binom{k^2}{k}n$$which, when inducted on, gives the nice bound $\text{ex}_P(n) \leq 2k^4\binom{k^2}{k}n$.

In Fox's work where the constant was improved \cite{fox_2011}, it also follows this framework but with some extra clever ideas below. Consider a $k \times k$ permutation matrix $P$. For integers $s \leq t$, let $f_P(t,s)$ be the maximum possible $N$ such that there exists an $N \times t$ matrix with at least $s$ ones in each row that avoids $P$. Similarly, $g_P(t, s)$ is the horizontal analog of $f_P(t, s)$. Then, Fox showed
\begin{equation} \label{eq:fox-ineq}
    \text{ex}_P(tn) \leq \text{ex}_P(s-1)\text{ex}_P(n) + \text{ex}_P(t)(f_P(t, s) + g_P(t, s))n
\end{equation}

using an argument involving contracting the rows/columns to view the distribution of one-entries globally. After this inequality, Fox bounded $f_P(t, s) = g_P(t, s) \leq \frac{2^{k-1}t^2}{s}$ for any values of $s \leq t$ using the following method: for an $N \times t$ matrix $M$ that avoids the all ones $r \times k$ matrix $J_{r, k}$, split the columns into two sections, both of width $t/2$. The number of rows where either the first or last $t/2$ entries are empty is at maximum $2f_P(t/2, s)$, and the remaining rows have at least one one-entry in both the first and last $t/2$ entries, and these rows can be bounded by $2f_{J_{r, k-1}}(t/2, s/2)$.  This gives the recursion $f_{J_{r, k}} \leq 2f_{J_{r,k}}(t/2, s) + 2f_{J_{r, k-1}}(t/2, s/2)$, which can be inducted on to get $f_P(t, s) \leq r2^{k-1}t^2/s$. Lastly, Fox picked suitable values for $(t, s)$ and did induction on $(\star)$ to obtain $\text{ex}_P(n) \leq 3k(2^{8k})n$ for any $k$-permutation matrix $P$. 

Fox's approach is meant to work for all $k \times k$ permutations $P$, which is why in the recursion we avoid an all ones matrix. For $c$-blockable permutations, we will improve the bound on $f_P(t, s)$ for this family using a new recursive approach. In this approach, we will take advantage of the unique structure of blockable permutations. In turn, by Inequality (\ref{eq:fox-ineq}), this improves the bound on the extremal constant for permutations in this family to polynomial in $k$. 

Before we can bound $f_P(t, s)$, we will need to show that the pre-condition of Theorem \ref{thm:mainresult}, namely that the extremal functions of certain permutation matrices are polynomially bounded in $k$, directly implies that $f_P(t, s)$ for these same matrices is also polynomially bounded in $k$.

\begin{lem} \label{lem: lemma2.1}
Consider an arbitrary permutation matrix $P$ with Furedi-Hajnal limit (also known as the extremal constant) $k^a$ for a positive constant $a$. Then for any positive integers $s, t$ such that $s \leq t$ and $s \geq k^a$, we have $$f_P(t, s) \leq \frac{k^at}{s-k^a}$$
\end{lem}

\begin{proof}
    Note that since the extremal constant is $c_P = k^a$, we know that for all positive integers $m, n$, we have $\text{ex}_P(m+n) \leq k^a(m+n)$. 
    
    For the sake of contradiction, assume there exists a $P$-avoiding matrix $M$ of size $r \times t$ such that each row has at least $s$ ones, where $r > \frac{k^at}{s-k^a}$. Note that in order to avoid $P$, the maximum number of ones $M$ can have is $k^a(r+t)$. However, in $M$, note that the minimum number of ones in these $r$ rows is $rs > k^a(r+t)$, which is a contradiction. 

    Hence, we must have $f_P(t, s) \leq \frac{k^at}{s-k^a}$, as desired.
\end{proof} 

Now, in Lemma \ref{lem: lemma2.2}, we will use a new recursive approach to obtain a improved bound on $f_P(t, s)$ where $P$ is a $c$-blockable permutation using a general recursion that works for any pair of positive integers $(t, s)$. Then, in Lemma \ref{lem: lemma2.3}, we will pick specific values for $t$ and $s$ and evaluate the recursion in that case. This improved polynomial in $k$ bound on $f_P(t, s)$ will lead to a polynomial in $k$ bound on the extremal constant too.

\begin{lem} \label{lem: lemma2.2}
Let $P$ be a $c$-blockable permutation of size $k$. For any positive integers $t, s$ such that $t \geq s$ and constants $0 < x, y < 1$, we have $$f_P(t, s) \leq  \binom{c}{\lfloor xc \rfloor}f_P\left(\left\lfloor t \cdot \frac{\lfloor xc \rfloor}{c}\right\rfloor, \lfloor sy \rfloor\right) + \frac{k^at}{s\left(1 - y\left(\frac{c-1}{\lfloor xc \rfloor}\right)\right)(c) - k^ac}$$

\end{lem}

\begin{proof}
Consider a $P$-avoiding matrix $M$ with $t$ columns and at least $s$ ones per row. Partition the columns into $c$ vertical {\it sections} of equal width, discarding at most $\frac{t}{c}$ columns if needed. So roughly each vertical section has $t/c$ columns. Additionally, consider the constants $0 < x, y < 1$ whose relevance we will describe soon (see Fig \ref{fig:prooffig} for an example of this setup where $c=5$ and $x=0.8$). We will show that $M$ cannot contain more rows than the stated bound. 
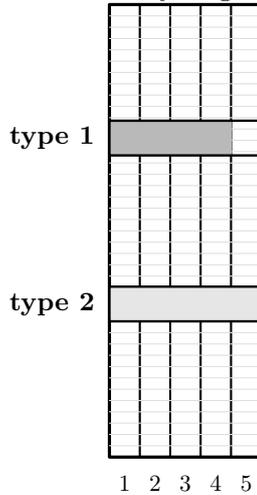
\begin{figure}[htbp]
\caption{Density Argument}
\label{fig:prooffig}
\centering
\begin{tikzpicture}[line cap=round,line join=round,>=latex, scale=0.2]

\def\W{10}
\def\H{30}
\def\c{5}
\def\seg{2}

\draw[line width=1.2pt] (0,0) rectangle (\W,\H);

\foreach \i in {1,...,4}{
  \draw[line width=0.8pt] (\i*\seg,0) -- (\i*\seg,\H);
}

\foreach \y in {0.75,1.5,...,29.25}{
  \draw[gray!25,line width=0.4pt] (0,\y) -- (\W,\y);
}

\def\yA{20.0}
\def\band{2.3}

\foreach \j in {0,1,2,3}{
  \fill[gray!55] (\j*\seg,\yA) rectangle ++(\seg,\band);
}
\draw[line width=0.9pt] (0,\yA) rectangle (\W,\yA+\band);
\node[anchor=east] at (-0.3,\yA+0.5*\band) {\small \textbf{type 1}};

 (\seg*4,\yA+\band+0.5);

\def\yB{9.0}

\foreach \j in {0,1,2,3,4}{
  \fill[gray!20] (\j*\seg,\yB) rectangle ++(\seg,\band);
}
\draw[line width=0.9pt] (0,\yB) rectangle (\W,\yB+\band);
\node[anchor=east] at (-0.3,\yB+0.5*\band) {\small \textbf{type 2}};

 (0,\yB-0.5) -- 
 (\W,\yB-0.5);

\foreach \j [evaluate=\j as \x using (\j-0.5)*\seg] in {1,...,\c}{
  \node[scale=0.8,anchor=north] at (\x,0) [below=4pt] {\j};
}
\end{tikzpicture}
\end{figure}

In this setup, each row $R$ is split into $c$ sections of equivalent width. We refer to these sections as 
$S_1(R), \dots, S_c(R)$, and we will often refer to whole row as the union set $S(R) = \{S_1(R), \dots, S_c(R)\}$. 

Call a row \textit{type $1$} if there exists a set of $\lfloor xc \rfloor$ (not neccessarily adjacent sections) of the $c$ sections in the row with at least $\lfloor ys \rfloor$ ones. Otherwise, the row is considered \textit{type $2$}. 

For a type $2$ row, note that in total we must have at least $s$ ones. However, since each set of $\lfloor xc \rfloor$ sections in this row has at most $\lfloor ys \rfloor$ ones, we know that each individual section must have at least \[s - sy \frac{\binom{c-1}{\lfloor xc \rfloor}}{\binom{c-2}{\lfloor xc \rfloor -1}} = s - sy\left(\frac{c-1}{\lfloor xc \rfloor}\right)\] ones.  This is because for each section, for the rest of the $c-1$ sections, the total number of ones the rest of the $c-1$ sections contain is at most $sy \frac{\binom{c-1}{\lfloor xc \rfloor}}{\binom{c-2}{\lfloor xc \rfloor -1}}$ by the definition of Type 2 and the fact that any $\lfloor xc \rfloor$ sections have less than $\lfloor ys \rfloor$ ones. 

Since every row in $M$ is either a type $1$ or a type $2$ row, we can bound the number of type $1$ rows and type $2$ rows separately under the constraint of avoiding $P$ and then their sum will be the final bound on $f_P(t, s)$. 

For each type $1$ row $R$, there are $\binom{c}{\lfloor xc \rfloor}$ distinct subsets of $S(R)$ with cardinality exactly $\lfloor xc \rfloor$ sections. Given that each type $1$ row must contain at least $\lfloor ys \rfloor$ ones in one of these subsets, an arbitrary subset cannot show up for more than $f_P(\frac{t}{c} \cdot \lfloor xc \rfloor, \lfloor sy \rfloor)$ type $1$ rows, otherwise $P$ will be contained by the definition of $f_P$. Hence, in total, we cannot have more than $\binom{c}{\lfloor xc \rfloor} f_P(\frac{t}{c} \cdot \lfloor xc \rfloor, \lfloor sy \rfloor)$ type $1$ rows.

Now, we only consider the rows of type $2$ in $M$, and let $M_1, \dots, M_c$ denote the submatrices formed by the $c$ vertical portions in $M$. Let $P_i$ be the $i^{\text{th}}$ block (by height) in the $c$-blockable matrix $P$. Each of these $P_i$ corresponds to a certain section $S_{j_i}(R)$.

Notice that $M_{j_i}$ will contain permutation matrix $P_i$, the $i^{\text{th}}$ block (by height) in the $c$-blockable matrix $P$, after $f_{P_i} \left(\frac{t}{c}, s\left(1 - y\left(\frac{c-1}{\lfloor xc \rfloor}\right)\right)\right)$ rows. Thus, we can have at most $$\sum_{i=1}^{c} f_{P_i}\left(\frac{t}{c}, s\left(1 - y\left(\frac{c-1}{\lfloor xc \rfloor}\right)\right)\right) \leq \left(\frac{\frac{t}{c}}{s\left(1 - y\left(\frac{c-1}{\lfloor xc \rfloor}\right)\right) - k^a}\right)(k^a)$$
type $2$ rows, because otherwise the first  $f_{P_1} \left(\frac{t}{c}, s\left(1 - y\left(\frac{c-1}{\lfloor xc \rfloor}\right)\right)\right)$ type $2$ rows will contain $P_1$ in $M_{j_1}$, the next  $f_{P_2} \left(\frac{t}{c}, s\left(1 - y\left(\frac{c-1}{\lfloor xc \rfloor}\right)\right)\right)$ type $2$ rows will contain $P_2$ in $M_{j_2}$, and so on until the last  $f_{P_c} \left(\frac{t}{c}, s\left(1 - y\left(\frac{c-1}{\lfloor xc \rfloor}\right)\right)\right)$ type $2$ rows containing $P_c$ in $M_{j_c}$, and hence $M$ would contain $P$. 

Given that each row of $M$ is either type $1$ or type $2$, and we have bounds for the number of type $1$ rows and the number of type $2$ rows, respectively, we know that
$$f_P(t, s) \leq  \binom{c}{\lfloor xc \rfloor}f_P\left(\left\lfloor t \cdot \frac{\lfloor xc \rfloor}{c}\right\rfloor, \lfloor sy \rfloor\right) + \frac{k^at}{s\left(1 - y\left(\frac{c-1}{\lfloor xc \rfloor}\right)\right)(c) - k^ac}$$ as desired.
\end{proof}

Now we iterate this recursion, letting the $t$-value and $s$-value after $i$ steps of the recursion be $t_i$ and $s_i$, respectively, the initial values being $(t_0, s_0)$. Ultimately, we would like this recursion to reach a final stage of $(t_{\text{final}}, s_{\text{final}}) = (\beta k \pm c_1, \beta k \pm c_2)$, where $\beta$ is a polynomial in $k$. However, there are three constraints that we must satisfy at each step of the recursion: we must have $t \geq s$ at each step for it to be logical, we must have $s\left(1-y\left(\frac{c-1}{\lfloor xc \rfloor}\right)\right) > k^a$ to maintain the constraint established by Lemma \ref{lem: lemma2.1}, and lastly we must have $x > \frac{1}{c}$ – otherwise, our definition for type 1 rows will be meaningless. 

In order to maintain these constraints while still ending the recursion at our desired values, we must be careful with the choice of our constants and the choice of $(t_0, s_0)$ here. We present a set of choices that satisfy all of these constraints.

First, set \[\beta = 2ck^{a-1}\] where $a$ is the constant from Lemma \ref{lem: lemma2.1}. Throughout the recursion, we intend to keep \[x = x_b = 1-\frac{1}{c}\] and \[y = y_b = 1-\frac{1}{2c}-\frac{1}{16c^2}\] for all of the steps other than the last two (here the subscript $b$ stands for ``beginning"). The steps other than the last two are called the \textit{bulk steps}, and we choose to have $R_A$ such steps. Accordingly, we will use $t_0 = \beta k\cdot x_{b}^{-\left(R_{A}+2\right)}$ and $s_0 = \sqrt{t_{0}}$. For the second to last step we will use a different $y$ value, namely $y_{1}$,   and the same $x$ value $x_b$. For the last step, we will simply use $x_b$ for both the $x$ and $y$ values. Throughout this process, we will also ignore the floor functions in the recursion from Lemma \ref{lem: lemma2.2}, later showing that they do not affect the order of magnitude of the bound: they only affecting the constants $c_1, c_2$ as shown in $(t_{\text{final}}, s_{\text{final}})$. We also set some milestones in the recursion for reference: let $(t_A, s_A)$ be the state after the first $R_A$ steps, let $(t_B, s_B)$ be the state after the next step, and let $(t_\text{final}, s_{\text{final}})$ be the state after that which is final.

\begin{lem} \label{lem: lemma2.3}
    Let $$R_A = \left \lceil 1+\frac{\ln\left(c\sqrt{\beta k}\right)}{\ln\left(\frac{y_b}{\sqrt{x_b}}\right)} \right \rceil$$ and let $y_1 = \frac{\beta k}{x_b s_A}$. If $(t_0, s_0) = (\beta k \cdot x_b^{-(R_A+2)}, \sqrt{t_0})$ as defined above, then $$f_P(t_0, s_0) \leq C_1k^{C_2}$$ for constants $C_1$ and $C_2$ in terms of $c$ and $a$.
\end{lem}

\begin{proof}
Before we show that our recursion gives this bound, we show that we satisfy all desired constraints. It is clear that $0 <x_b, y_b < 1$ by definition. Additionally, note that $y_1 < 1$ because $$y_1 = \frac{\beta k / x_b}{s_A} = \frac{\beta k / x_b}{s_0 y_b^{R_A}} = \frac{\beta k / x_b}{\beta k \, x_b^{-(R_A+2)} y_b^{R_A}} = x_b^{R_A+1} y_b^{-R_A},$$ and since $(\sqrt{x_b} / y_b)^{R_A - 1} \le \frac{1}{c \sqrt{\beta k}}$, we have $y_1 = x_b^2 \cdot (\sqrt{x_b}/y_b)^{2(R_A-1)} \le x_b^2 \cdot \frac{1}{c^2 \beta k^2} < 1$. In addition, note that $x_b > \frac{1}{c}$, which was one of our constraints.

Next, note that $s(1-y) > k^a$ at each recursion step. This is because for the bulk steps, $$s(1-y) \geq \beta k(1-y_b) > \beta k \cdot \frac{1}{2c}=k^a$$ and the condition is also true for the last two steps since $y_1 < y_b$. 

Now, we show that our recursion ends in the desired state. Note that after we run $R_A$ steps of the recursion using the constants $(x_b, y_b)$, our $t$ and $s$ values are $$(t_A, s_A) = \left(\beta kx_b^{-2}, \sqrt{\beta k}x^{-(R_A/2+1)}y^{R_A}\right)$$ Then, in the second to last step when we use the constants $(x_b, y_1)$, we have $$(t_B, s_B) = (\beta kx_b^{-1}, \beta kx_b^{-1})$$Lastly, after using the constants $(x_b, x_b)$, we achieve $(t_{\text{final}}, s_{\text{final}}) = (\beta k, \beta k)$. 

Now we will finally show our attained bound. We prove the bound by induction on the recursion steps. For each step $i$, let $(t_i, s_i)$ denote the $t$ and $s$ values at that step. The recursion is $$f_P(t_i, s_i) \le c \, f_P(t_{i+1}, s_{i+1}) + \frac{k^a t_i}{s_i (1-y_i) c - k^a c}$$ where $y_i = y_b$ for bulk steps and $y_i = y_1$ for second-to-last step. Also, $x=x_b$ is fixed throughout. 

For our induction hypothesis, we claim that for any step $i$, $$f_P(t_i, s_i) \le c^{R_A+2-i} f_P(t_{\text{final}}, s_{\text{final}}) + \sum_{j=i}^{R_A+1} c^{j-i} \frac{k^a t_j}{s_j (1-y_j) c - k^a c}$$ The base case is at the last recursion step, namely at $i=R_A+1$, where we have $$f_P(t_{R_A+1}, s_{R_A+1}) \le c f_P(t_{\text{final}}, s_{\text{final}}) + \frac{k^a t_{R_A+1}}{s_{R_A+1} (1-y_{R_A+1}) c - k^a c},
$$ which matches the induction hypothesis. Now, assume the bound holds for step $i+1$. Then by Lemma \ref{lem: lemma2.2}, we have $$f_P(t_i, s_i) \le c f_P(t_{i+1}, s_{i+1}) + \frac{k^a t_i}{s_i (1-y_i) c - k^a c}$$ and plugging in the inductive hypothesis for $f_P(t_{i+1}, s_{i+1})$ gives us 
\begin{align*}
f_P(t_i, s_i) &\le c \Bigg[c^{R_A+1-i} f_P(t_{\text{final}}, s_{\text{final}}) + \sum_{j=i+1}^{R_A+1} c^{j-(i+1)} \frac{k^a t_j}{s_j (1-y_j) c - k^a c} \Bigg] + \frac{k^a t_i}{s_i (1-y_i) c - k^a c} \\
&= c^{R_A+2-i} f_P(t_{\text{final}}, s_{\text{final}}) + \sum_{j=i}^{R_A+1} c^{j-i} \frac{k^a t_j}{s_j (1-y_j) c - k^a c}. 
\end{align*}
and this completes our induction. So, the bound will also hold for $i=0$ (the final step): 
\begin{equation}
f_P(t_0, s_0) \le c^{R_A+2} f_P(t_{\text{final}}, s_{\text{final}}) + \sum_{j=0}^{R_A+1} c^j \frac{k^a t_j}{s_j (1-y_j)c - k^a c} \label{eq:1}
\end{equation}
and thus, a crude bound is
\begin{equation} \label{eq:fpts_ineq}
    f_P(t_0, s_0) \le c^{R_A+2} f_P(t_{\text{final}}, s_{\text{final}}) + (2c)^{R_A+2} \frac{k^a t_0}{s_0 (1-y) c - k^a c}
\end{equation}

Note that we can directly reach this crude bound (\ref{eq:fpts_ineq}) because the second additive term in (\ref{eq:1}), which is $c^j \frac{k^a t_j}{s_j (1-y_j)c - k^a c}$, only at most doubles as the recursive steps progress. In the bulk steps, this is because if $A_j = \frac{k^at_j}{s_j(1-y_j)c - k^ac}$, and recall that $t_j = t_0x_b^j$  and $s_j = s_0y_b^j$, then $$\frac{A_{j+1}}{A_j} = \frac{x_b t_j}{y_b s_j (1-y_{b})c - k^a c} \cdot \frac{s_j (1-y_b)c - k^a c}{t_j} = \frac{x_b \bigl(s_j(1-y_b) - k^a\bigr)}{y_b s_j (1-y_{b}) - k^a} \le 2,$$  where the last inequality is true because it is equivalent to $(2y_b-x_b)s_j(1-y_b) \geq (2-x_b)k^a.$ Since $2y_b-x_b = 1-\frac{1}{8c^2}$ and $2-x_b=1+\frac{1}{c}$, this is equivalent to $$\left(1-\frac{1}{8c^2}\right)s_j\left(\frac{1}{2c}+\frac{1}{16c^2}\right) \geq k^a\left(1+\frac{1}{c}\right),$$
or equivalently $s_j \geq k^a \cdot \frac{128(c+1)c^3}{(8c^2-1)(8c+1)}$, which we know is true since $$s_j\geq \beta ky_b^{-2} \geq 2ck^a\left(1-\frac{1}{2c}-\frac{1}{16c^2}\right)^{-2} = \frac{512c^{5}}{\left(16c^{2}-8c-1\right)^{2}}k^a\geq \frac{128\left(c+1\right)c^{3}}{\left(8c^{2}-1\right)\left(8c+1\right)}k^a.$$
 For the second-to-last step we can show something even stronger as $$\frac{A_{R_A}}{A_0} = \frac{k^a t_{R_A}}{s_{R_A} (1-y_1)c - k^a c} \cdot \frac{s_0 (1-y_b)c - k^a c}{k^a t_0}  
= \frac{x_b^{R_A} (s_0 (1-y_b) - k^a / c)}{s_0 y_b^{R_A} (1-y_1) - k^a c} 
\le \frac{(x_b / y_b)^{R_A} (1-y_b)}{1-y_1} \le 1,
$$ 
where the last inequality is true because $y_1=\frac{\beta k}{s_Ax_b}=\frac{y_b^2}{x_b}$ and thus $1-y_1=\frac{x_b-y_b^2}{x_b} \geq \frac{1}{4c}$, also $1-y_b \leq \frac{1}{c}$, and hence $\frac{1-y_b}{1-y_1} \leq 4$. Additionally, $\ln(x_b/y_b)=\ln(1-\frac{y_b-x_b}{y_b}) \leq -\frac{y_b-x_b}{y_b} \leq -\frac{1}{3c}$, so $(x_b/y_b)^{R_A} \leq e^{-R_A/3c}$, and since $R_A \geq \frac{ac^2 \ln k}{2}$, this gives $(x_b/y_b)^{R_A} \leq k^{-ac/6}$. Combining we get $(x_b/y_b)^{R_A} \frac{1-y_b}{1-y_1} \leq 4k^{-ac/6} \leq 1$ as desired. 
Similar inequality holds for the last step $A_{R_A+1}$. Hence, every $A_j \leq A_0$.

Note that throughout this process, we have ignored the floor functions in our recursion. However, they are not hard to account for: notice that if we floor both the $t$ and $s$ values for each of our $R_A+2$ recursion steps, the $t_{\text{final}}$ value will be at most $\beta k$, and the $s_{\text{final}}$ value will be between $\beta k - \frac{1}{1-y}$ and $\beta k$. 
Observe that a crude bound for $f_P(t_{\text{final}}, s_{\text{final}})$ is $k\binom{t_\text{final}}{s_{\text{final}}}$ due to the Pigeonhole Principle – after these many rows, there must be $k$ rows in $M$ with ones in exactly the same $s_{\text{final}} > k$ columns.

Hence, 
$$f_P(t_0, s_0)\leq c^{R_A+2} \binom{t_{\text{final}}}{s_\text{final}} + (2c)^{R_A+2}\left(\frac{k^at_0}{s_0(1-y)(c)-k^ac}\right)$$ 

$$\leq c^{R_A+2}k\binom{\beta k}{\beta k - \frac{1}{1-y}} + (2c)^{R_A+2}\left(\frac{k^at_0}{s_0(1-y)(c)-k^ac}\right)$$

Note that $\frac{1}{1-y}$ is a constant, so $\binom{\beta k}{\beta k-\frac{1}{1-y}}=\binom{\beta k}{O(1)}=k^{O(1)}$. Since $R_A=O(\log k)$, we have $c^{R_A+2}=k^{O(1)}$, and the second term is also polynomial in $k$. Therefore $f_P(t_0,s_0)\le C_1 k^{C_2}$ for constants $C_1,C_2$.
\end{proof}

Recall Fox's definition of $g_P(t, s)$ as the horizontal analog of $f_P(t, s)$. We can interpret it as $f_{P'}(t, s)$ where $P'$ is a permutation matrix formed by rotating a $c$-blockable permutation matrix $P$. Since $P'$ itself is also a $c$-blockable permutation matrix, we now also know that $g_P(t, s) \leq C_1 k^{C_2}$.  

Now, we use this result for Fox's inequality to obtain a polynomial bound on the extremal number of all $c$-blockable permutations.

\begin{thm} \label{thm: thm2.4}
For a $c$-blockable $k$-permutation $P$, we have
    $$\mathrm{ex}_P(n) \leq k^{\alpha}n$$ where $\alpha = 2a+8c^2+32ac^2 \ln c$.
\end{thm}

\begin{proof}
We use induction. For the base cases, clearly the statement is true for $n \leq k^{\alpha}$. 

Plugging in our bounds on $f_P(t, s)$, and equivalently $g_P(t, s)$, applied to inequality (\ref{eq:fpts_ineq}) and (\ref{eq:fox-ineq})  we have
$$\mathrm{ex}_P(s^2n) \leq (s-1)^2\mathrm{ex}_P(n) + 2\left(c^{R_A+2}k\binom{\beta k}{\beta k - \frac{1}{1-y}} + (2c)^{R_A+2}\left(\frac{k^at}{s(1-y)c-k^ac}\right)\right)$$

Recall that $a$ is the exponent from the hypothesis $\text{ex}_{P_i}(n) \leq k^an$ in Lemma \ref{lem: lemma2.1}, and that $\beta=2ck^{a-1}$, $x=x_b=1-\frac{1}{c}$, $y=y_b=1-\frac{1}{2c}-\frac{1}{16c^2}$, and lastly $$R_A = \left \lceil 1+\frac{\ln\left(c\sqrt{\beta k}\right)}{\ln\left(\frac{y_b}{\sqrt{x_b}}\right)} \right \rceil = \Theta(c^2 \ln(c\sqrt{\beta k})) = \Theta(c^2(O(\ln c) + \frac{a}{2} \ln k)) \geq \frac{ac^2 \ln k}{2}$$
Additionally, recall that $$s = \sqrt{\beta kx_b^{-(R_A+2)}} \geq \sqrt{2ck^a e^{R_A \ln(1/x_b)}} \geq \sqrt{2ck^ae^{\frac{ac^2 \ln k}{2}(1/c + 1/2c^2)}} \geq \sqrt{k^ak^{ac/2}} \geq k^{ac/4}$$
These lower bounds are estimates we will use in our future computations. 

Assuming that the theorem statement is true for all $n < n'$, consider the case $n = n'$, where in order for our induction to be complete, we need $$(s-1)^2 \cdot k^\alpha + 2\left(c^{R_A+2}k\binom{\beta k}{\beta k - \frac{1}{1-y}} + (2c)^{R_A+2}\left(\frac{k^{a}t}{s(1-y)c-k^{a}c}\right)\right) \leq s^2 \cdot k^\alpha$$ or equivalently $$ (2c)^{R_A+2}\left(k\binom{\beta k}{\beta k - \frac{1}{1-y}} + \left(\frac{k^{a}t}{s(1-y)c-k^{a}c}\right)\right) \leq (2s-1)k^\alpha$$
for which a close-to-optimal bound is $$\alpha = 2a + 8c^2+32ac^2 \ln c$$
which works because 
\begin{align*} 
&(2c)^{R_A+2} \Biggl( k \binom{\beta k}{\beta k - \frac{1}{1-y}} + \frac{k^a t}{s(1-y)c - k^a c} \Biggr) \notag\\
&\le2^{R_A+1}\left(2c^{R_A+2} \Biggl( k \binom{\beta k}{\beta k - \frac{1}{1-y}} + \frac{k^a t}{s(1-y)c - k^a c} \Biggr)\right) \notag\\
&\le 2^{R_A+1} \left(2c^{R_A+2} \cdot k(\beta k)^{\frac{1}{1-y}} + 2c^{R_A+2} \frac{k^as^2}{c(s/3)}\right) \quad \text{(as } s \geq k^{ac/4} \geq k^ac^{c/4} \geq \frac{3k^a}{2-3y} \implies s(1-y)-k^a \geq \frac{s}{3})\notag\\
&= 2^{R_A+1}\left(2c^{R_A+2} \cdot k (\beta k)^{\frac{1}{1-y}} + 6c^{R_A+2} \cdot k^a s\right)  \notag\\
&= 2^{R_A+1}\left(2c^{R_A+2} \cdot k (2c)^{\frac{1}{1-y}} k^{\frac{a}{1-y}} + 6c^{R_A+2} \cdot k^a (\beta k x^{-(R_A+2)})^{1/2}\right) \notag\\
&\le 2^{R_A+1}\left(k^{2c + 2 + \frac{8}{3}ac^2 \ln c + a(2c)} + k^{4c^2 + a + 8ac^2 \ln c}\right) \quad (\text{due to }  c^{R_A+2} \leq e^{\left(
\frac{\ln c\,\ln(c\sqrt{\beta k})}{\ln(y/\sqrt{x})}
\right)} \text{ , and } x^{-(R_A+2)}\le k^{2ac}) \notag\\
&\le 2^{R_A+1} (2k^{\alpha/2}) \notag\\
&\le 2s^c (2k^{\alpha/2}) \quad (\text{due to } \ln(2s^c) = \ln(2) + \frac{c}{2}\ln(\beta k) + \frac{c}{2}(R_A+2)(-\ln x_b) \geq R_A/2 \implies 2^{R_A+1} \leq 2s^c)\notag\\ 
&\le (2s-1) k^{\alpha} \quad (\text{due to } s=\sqrt{2c}k^{a/2}x_b^{-(R_A+2)/2} \text{, } x_b^{-(R_A+2)/2} \leq k^{4c \ln c} \implies s^{c-1}\leq k^{\alpha/2}) \nonumber
\end{align*}
\end{proof}
Hence, the Furedi-Hajnal limit (also known as the extremal constant) is $c(\pi) \leq k^{2a + 8c^2+32ac^2 \ln c}$. By Cibulka's equivalence $L(\pi) = O(c(\pi)^2)$, we know that $$L(B_{\{\mathcal{P}_1, \dots, \mathcal{P}_c\}}) \leq O(k^{4a + 16c^2+64ac^2 \ln c})$$
Since the above bound grows polynomially in $k$, Theorem \ref{thm:mainresult} follows.

\end{document}